\documentclass[11pt]{amsart}
\usepackage{url}
\usepackage{amsmath}
\usepackage{graphicx}
\usepackage{amssymb}
\usepackage{mathtools}
\usepackage[T1]{fontenc}
\usepackage{tikz}
\usepackage{enumerate}

\newcommand*{\rom}[1]{\expandafter\@slowromancap\romannumeral #1@}
\newcommand{\RNum}[1]{\uppercase\expandafter{\romannumeral #1\relax}}
\newcommand{\N}{\mathbb{N}}  % positive integers
\newcommand{\Z}{\mathbb{Z}}  % integers
\newcommand{\R}{\mathbb{R}}  % real numbers
\newcommand{\Q}{\mathbb{Q}}
\newcommand{\acts}{\curvearrowright}

\newcommand{\Rat}{\mathrm{Rat}}
\newcommand{\Eig}{\mathrm{Eig}}
\newcommand{\denom}{\mathrm{denom}}
\newcommand{\lcm}{\mathrm{lcm}}

\newcommand\labs[1]{\left\lvert#1\right\rvert} 
\newcommand\norm[1]{\lVert#1\rVert}
\newcommand\lnorm[1]{\left\lVert#1\right\rVert}
\newcommand{\T}{\mathbb{T}}

\theoremstyle{plain} 
\newtheorem{thm}{Theorem}[section]
\newtheorem{lem}[thm]{Lemma}
\newtheorem{pro}[thm]{Proposition}

\newtheorem{qu}{Question}

\newtheorem{thmAlph}{Theorem}[section]

\theoremstyle{definition} 
\newtheorem{exam}[thm]{Example}
\newtheorem{defn}[thm]{Definition}

\theoremstyle{remark} 

\newtheorem*{ack}{Acknowledgements}

\title{Quantitative Expansivity for ergodic $\Z^d$-actions}
\author{Alexander Fish and Sean Skinner}

\address{School of Mathematics and Statistics, University of Sydney, Australia}
\curraddr{}
\email{alexander.fish@sydney.edu.au}
\email{sean.skinner@sydney.edu.au}
\thanks{}

\date{\today}                                           % Activate to display a given date or no date

\begin{document}
\begin{abstract}
We study expansiveness properties of positive measure subsets of ergodic $\Z^d$-actions along two different types of structured subsets of $\Z^d$, namely, cyclic subgroups and images of integer polynomials. We prove quantitative expansiveness properties in both cases and strengthen combinatorial results obtained by Björklund and Fish in \cite{BjF} and Bulinski and Fish in \cite{BuF}.
Our methods unify and strengthen earlier approaches used in \cite{BjF} and \cite{BuF} and to our surprise, also yield a counterexample to a certain pinned variant of the polynomial Bogolyubov theorem.
\end{abstract}

\maketitle
\section{\textbf{Introduction}}
An influential result of Furstenberg, Katznelson and Weiss \cite{FKW} states that if $A\subset \R^2$ has positive upper density with respect to the Lebesgue measure $m$, i.e.
\[ \lim_{N\to \infty} \frac{m(A\cap [-N,N]^2)}{m([-N,N]^2)}>0,\]
then the set of all distances between pairs of points in $A$ satisfies 
\[[m_0,\infty) \subset \{|x-y|\, : \, x,y\in A\}\]
for some $m_0=m_0(A)>0$. In \cite{Mag} Magyar established a discrete analogue of this result for sets of positive upper Banach density in $\Z^d$. Recall that the upper Banach density of a set $E\subset \Z^d$ is defined to be
\begin{equation*}
d^*(E):= \lim_{N\to \infty} \sup_{t \in \Z^d} \frac{|E\cap (Q_N+t)|}{|Q_N|}
\end{equation*}
where $Q_N:= [-N,N]^d \cap \Z^d$.
\begin{thm}[Quantitative distances \cite{Mag}]\label{Theorem: Unpinned distances}
Let $d \geq 5$ be a positive integer. Then for all $E\subset \Z^d$ with $d^*(E)>0$ there exist some positive integers $k=k(d^*(E))$ and $m_0 = m_0(E)$ such that
\begin{equation*}
km \in \{|x-y|^2 \, : \, x,y\in E\} \quad \text{for all integers }m\geq m_0.
\end{equation*}
\end{thm}
The term \textit{quantitative} in the title of Theorem \ref{Theorem: Unpinned distances} refers to the fact that the integer $k$ depends only on $d^*(E)$ and not on the set $E$ itself. In \cite{LyMag} Lyall and Magyar went on to prove a strengthened, \emph{pinned} variant of Theorem \ref{Theorem: Unpinned distances}.
\begin{thm}[Quantitative pinned distances \cite{LyMag}] \label{Theorem: Pinned distances}
Let $d \geq 5$ be a positive integer. Then for all $E\subset \Z^d$ with $d^*(E)>0$ there exists some positive integers $k=k(d^*(E))$ and $m_0 = m_0(E)$ such that for every $m_1\geq m_0$ there exists a fixed point $x \in E$ such that
\begin{equation*}
km \in \{|x-y|^2 \, : \, y\in E\} \quad \text{for all integers } m_0\leq m \leq m_1.
\end{equation*}
\end{thm}
In a series of works by Björklund, Bulinski and Fish \cite{BjF,BuF,BjBu}, it was realised that similar results hold if one replaces the squared Euclidean distance with other functions. We will focus on two of these results.
\begin{thm}[Quantitative polynomial Bogolyubov theorem \cite{BuF}]\label{Theorem: Bogolyubov}
Let $P:\Z \to \Z$ be an integer polynomial with zero constant term and $\mathrm{deg}(P)\geq 2$. Then for every $\delta>0$ there exists a positive integer $k_0=k_0(P,\delta)$ such that the following holds. For every $E\subset \Z$ with upper Banach density $d^*(E)\geq \delta$ there exists a positive integer $k\leq k_0$ with
\[k\Z \subset E-E + P(E-E).\]
\end{thm}
\begin{thm}[Non-quantitative simplicies \cite{BjF}]\label{Theorem: Non quant Triangles}
Let $d\geq 2$ be an integer. For every $E\subset \Z^d$ with upper Banach density $d^*(E)>0$ there exists some positive integer $k=k(E)$ such that the set of all signed volumes of $d$-simplicies whose vertices are in $E$ contains the set $k\Z$.
\end{thm}
Three natural questions arise. Firstly, does a quantitative version of Theorem \ref{Theorem: Non quant Triangles} hold? Secondly, does a pinned variant of Theorem \ref{Theorem: Bogolyubov} hold? Thirdly, does a pinned variant of Theorem \ref{Theorem: Non quant Triangles} hold?
There is some ambiguity in the phrase \emph{pinned variant}, so let us be more precise.
\begin{qu}\label{Question: Quant triangles}
Can one ensure that the integer $k$ in Theorem \ref{Theorem: Non quant Triangles} depends only on $d^*(E)$ and not the set $E$ itself.
\end{qu}
\begin{qu}\label{Question: Pinned polynomial Bogolyubov theorem}
Let $P:\Z \to \Z$ be an integer polynomial with $P(0)=0$ and $\mathrm{deg}(P)\geq 2$. Is it true that for every $E\subset \Z$ with $d^*(E)>0$ there exists some positive integer $k$ such that for every positive integer $m$ there exist some $x,y\in E$ such that
\[\{-km,-k(m-1),\ldots,k(m-1),km\}\subset E-x + P(E-y) ?\]
\end{qu}
\begin{qu}\label{Question: Pinned Triangles}
Let $d\geq 2$ be an integer and suppose $E\subset \Z^d$ has $d^*(E)> 0$. For a point $x\in E$ denote by $\mathrm{VolSpec}_d(E,x)$ the set of all signed volumes of $d$-simplicies with vertex set $V$ satisfying that $x\in V$ and that $V\subset E$. Must there exist some positive integer $k$ such that for every finite subset $F\subset \Z$ there exists a point $x\in E$ with
\[ kF \subset \mathrm{VolSpec}_d(E,x)?\]
\end{qu}
In this paper we show that the answer to Question \ref{Question: Quant triangles} is \emph{yes} and that the answer to Question \ref{Question: Pinned polynomial Bogolyubov theorem} is \emph{no}. Question \ref{Question: Pinned Triangles} remains open.

As is now routine in density Ramsey theory, our combinatorial results, i.e. those about positive density subsets of $\Z^d$, are obtained by first proving analogous recurrence statements in the context of measure preserving $\Z^d$-actions and then translating these dynamical statements into combinatorial statements via the means of Furstenberg's correspondence principle. In particular, we use the following ergodic version of Furstenberg's correspondence principle. Recall that a measure preserving $\Z^d$-action $T:\Z^d \acts (X,\mu)$ on a probability space $(X,\mu)$\footnote{We choose not to include the underlying $\sigma$-algebra in our notation and moving forward all considered subsets of a measurable space will be assumed to be measurable.} is ergodic if every set $A\subset X$ satisfying $\mu(T^v A) = \mu(A)$ for all $v\in \Z^d$ has $\mu$-measure equal to $0$ or $1$.
\begin{pro}[Furstenberg's Correspondence Principle {\cite{BM}[Theorem 2.8]}]\label{Prop: FCP}
Let $E\subset \Z^d$ have $d^*(E)>0$. Then there exists an ergodic action $T: \Z^d \acts (X,\mu)$ and a set $A\subset X$ with $\mu(A) = d^*(E)$ satisfying that
\begin{equation}\label{eq: FCP conclusion}
\mu \left(\bigcap_{v\in F} T^v A\right) \leq d^*\left(\bigcap_{v\in F} (E+v)\right) \quad \text{for every finite } F\subset \Z^d.
\end{equation}
\end{pro}
Our main new dynamical contributions are two \emph{expansivity theorems} for ergodic $\Z^d$-actions, the first of which is a quantitative strengthening of the notion of \emph{directional expansiveness} as introduced in \cite{BjF} by Björklund and the first author. For us, a direction in $\Z^d$ is a cyclic subgroup generated by a primitive\footnote{By primitive we mean that the greatest common divisor of all of the components of $v$ is equal to $1$.} vector in $\Z^d$. The term \emph{directional} then refers to properties of the sub-action of some direction in $\Z^d$.

A first natural directional question to ask is whether or not every ergodic action $T: \Z^d \acts (X,\mu)$ admits some direction for which the directional sub-action is ergodic. The answer to this question is no, and amongst other things, Robinson Jr, Rosenblatt and Sahin in \cite{RRS} provide an example of a weak-mixing $\Z^d$-system which admits no ergodic directions.

Notice that if some direction $v\in \Z^d$ was ergodic for an action $T: \Z^d \acts (X,\mu)$, then every positive measure set $A\subset X$ would satisfy that
\[ \mu\left(\bigcup_{n\in \Z} T^{nv}A \right)=1.\]
In light of this observation and the negative answer provided by the authors of \cite{RRS} to the aforementioned question regarding ergodic directions, in \cite{BjF} Björklund and the first author asked instead if for every $\varepsilon>0$ and every positive measure set $A\subset X$ must there exist some direction $v\in \Z^d$ for which
\[ \mu\left(\bigcup_{n\in \Z} T^{nv}A \right)>1-\varepsilon?\]
Again the answer is no as shown by the following example from \cite{BjF}.
\begin{exam}[A set which is not directionally expandable]\label{Example: Set which is not directionally expandable}
For some integer $N\geq 2$, equip the space $X:= \Z^d/(N\Z)^d$ with the counting probability measure $\mu$. The action $T:\Z^d \acts (X,\mu)$ by translations preserves $\mu$, however for any singleton $A=\{x\}\subset X$ and any vector $v\in \Z^d$,
\[ \bigcup_{n\in \Z}T^{nv} A\]
is a coset of a cyclic subgroup of $X$, and so must have $\mu$-measure at most $1/N^{d-1}$.
\end{exam}
However, as was the central to their proof of Theorem \ref{Theorem: Non quant Triangles}, the authors of \cite{BjF} showed that highly expansive directions can always be found provided that one first passes to some suitable ergodic component of the sub-action of $k\Z^d$, for some $k$ depending on the set $A$ and on $\varepsilon$.  As eluded to earlier, our first expansivity theorem is a quantitative strengthening of this observation.
To state the theorem precisely, we require the notion of a $T^k$-ergodic component.
\begin{pro}[$T^k$-ergodic components {\cite{Bu}[Proposition A.2]}]\label{Prop: Tk ergodic components}
Let $T: \Z^d \curvearrowright (X,\mu)$ act ergodically. For any positive integer $k$ there exist finitely many $k\Z^d$-invariant and ergodic probability measures $\nu_1,\ldots,\nu_n$ with disjoint supports such that
\begin{equation*}
\mu = \frac{1}{n}\sum_{i=1}^n \nu_i.
\end{equation*}
Moreover each $\nu_i$ is of the form
\[ \nu_i(\cdot) = \frac{\mu(\cdot \cap C_i)}{\mu(C_i)}\]
for some $k\Z^d$-invariant set $C_i\subset X$.
We call $\nu_1,\ldots,\nu_n$ the $T^k$-ergodic components of $\mu$.
\end{pro}
\begin{thmAlph}[Quantitative directional expansivity]
\label{Theorem: Quant Direct Exp}
For every $\delta>0$ and $\varepsilon>0$ there exists some positive integer $k_0=k_0(\delta,\varepsilon)$ such that the following holds. For every ergodic action $T: \Z^d \curvearrowright (X,\mu)$ and every $A\subset X$ with $\mu(A)\geq \delta$ there exists some positive integer $k\leq k_0$, some $T^k$-ergodic component $\nu$ of $\mu$ with $\nu(A)\geq \mu(A)$, and some primitive vector $v\in \Z^d$ such that
\[ \nu\left( \bigcup_{n\in \Z} T^{n v}A \right) > 1-\varepsilon.\]
\end{thmAlph}
We remark that a non-quantitative version of Theorem \ref{Theorem: Quant Direct Exp} is implicit in \cite{BjF}, where $k$ and $\nu$ depend on $A$ and $\varepsilon$.
The affirmative answer to Question \ref{Question: Quant triangles} can then be deduced from Theorem \ref{Theorem: Quant Direct Exp} via the means of Proposition \ref{Prop: FCP}, and the details are provided in Section \ref{Section: Deducing combinatorial theorems}.
\begin{thmAlph}[Quantitative simplicies]\label{Theorem: Quantitative triangles}
Let $d\geq 2$ be an integer. For every $\delta>0$ there exists a positive integer $k_0 = k_0(\delta)$ such that the following is true. For every $E\subset \Z^d$ with upper Banach density $d^*(E)\geq \delta$ there exists some positive integer $k\leq k_0$ such that the set of all signed volumes of $d-$simplicies whose vertices are in $E$ contains the set $k\Z$.
\end{thmAlph}
Our proof of Theorem \ref{Theorem: Quantitative triangles} shares much in common with the proof of Theorem \ref{Theorem: Non quant Triangles} in \cite{BjF}, however the use of Theorem \ref{Theorem: Quant Direct Exp} both shortens and strengthens a key part of the proof.

The main new idea in the proof Theorem \ref{Theorem: Quant Direct Exp} is to use a new measure increment argument which is a direct measure theoretic analogue of the original density increment argument used by Roth \cite{Roth} in the proof of his famous theorem on three-term arithmetic progressions. The details of this measure increment argument are discussed in Section \ref{Section: Measure Increment}. A different type of measure increment argument was used in \cite{BuF} by Bulinski and the first author in their proof of Theorem \ref{Theorem: Bogolyubov}, and our measure increment argument also allows us to establish an expansivity theorem in this polynomial setting. In fact, we prove a multivariable polynomial expansivity theorem.
\begin{thmAlph}[Quantitative polynomial expansivity]\label{Theorem: Quant Poly Exp}
Let $P=(P_1,\ldots,P_d):\Z^r \to \Z^d$ be an integer polynomial in $r$ variables with zero constant term such that the component polynomials $P_1,\ldots,P_d$ are linearly independent. Then for every $\delta>0$ and every $\varepsilon>0$ there exists some positive integer $k_0=k_0(P,\delta,\varepsilon)$ such that the following holds. For every ergodic action $T: \Z^d \curvearrowright (X,\mu)$ and every $A\subset X$ with $\mu(A)\geq \delta$ there exists some positive integer $k\leq k_0$ and some $T^k$-ergodic component $\nu$ of $\mu$ with $\nu(A)\geq \mu(A)$ satisfying that
\[ \nu\left( \bigcup_{n\in \Z^r} T^{P(n)}A \right) > 1-\varepsilon.\]
\end{thmAlph}
From Theorem \ref{Theorem: Quant Poly Exp} we are able to prove a multidimensional extension of Theorem \ref{Theorem: Bogolyubov}.
\begin{thmAlph}[Quantitative multi-dimensional polynomial Bogolyubov theorem]\label{Theorem: Multidimensional Bog}
Let $P=(P_1,\ldots,P_d):\Z^d \to \Z^d$ be an integer polynomial in $d$-variables with zero constant term satisfying that no non-trivial linear combination of its component polynomials $P_1,\ldots,P_d$ has degree less than $2$. Then for every $\delta>0$ there exists a positive integer $k=k(P,\delta)$ such that the following holds. For every $E\subset \Z^d$ with upper Banach density $d^*(E)\geq \delta$ we have that
\[ k\Z^d \subset E-E + P(E-E).\]
\end{thmAlph}
We remark that the degree requirements in Theorems \ref{Theorem: Multidimensional Bog} and Theorem \ref{Theorem: Bogolyubov} are both necessary, and we prove this fact in Section \ref{Section: Appendix}.

Our proofs of Theorems \ref{Theorem: Quant Direct Exp} and \ref{Theorem: Quant Poly Exp} share several techniques with the results they extend from \cite{BjF} and \cite{BuF} respectively, however one of the central achievements of this paper is the synthesis of the ideas of expansivity developed in \cite{BjF} along with the measure increment techniques studied in \cite{BuF}. In particular, this unification yields an extension of the notion of expansivity for polynomial orbits in $\Z^d$.

In addition, the change in perspective provided by the use of Theorem \ref{Theorem: Quant Poly Exp} also allows us to establish a counter example to the pinned version of the polynomial Bogolyubov theorem, providing the negative answer to Question \ref{Question: Pinned polynomial Bogolyubov theorem}.

Indeed, as will be made clear from the deductions of Theorems \ref{Theorem: Quantitative triangles} and \ref{Theorem: Multidimensional Bog} from Theorems \ref{Theorem: Quant Direct Exp} and \ref{Theorem: Quant Poly Exp} respectively, pinned variants of both Theorems \ref{Theorem: Non quant Triangles} and \ref{Theorem: Bogolyubov} would follow if one could first establish strengthened versions of Theorems \ref{Theorem: Quant Direct Exp} and \ref{Theorem: Quant Poly Exp} in which one can take $\varepsilon=0$. In Section \ref{Section: Counter examples} we provide examples to show that both of these strengthenings fail. To our surprise, our counter example to the $\varepsilon=0$ version of Theorem \ref{Theorem: Quant Poly Exp} also yields a counter-example to the pinned version of the polynomial Bogolyubov's theorem described in Question \ref{Question: Pinned polynomial Bogolyubov theorem}. Indeed, in Section \ref{Section: Counter example to Bog} we prove the following.
\begin{thmAlph}[Counter-example to the pinned version of the polynomial Bogoylubov theorem]\label{Theorem: Counter-example to pinned Bog}
Let $P\in \Z[n]$ have $P(0)=0$ and $\deg P \geq 2$.
There exists a set $E\subset \Z$ with $d^*(E)>0$ such that for every positive integer $k$, there exists a positive integer $m$ with
\[ \{k,2k,\ldots, km\} \not \subset E-x + P(E-y) \quad \text{for every }x,y\in E.\]
\end{thmAlph}
\begin{ack}
A. Fish was supported by the ARC via grants DP210100162 and DP240100472. We are grateful to Nick Bridger for enlightening discussions on the topic of permutation polynomials and providing the reference to \cite{Turnwald}.
\end{ack}
\section{\textbf{Deduction of combinatorial theorems}}\label{Section: Deducing combinatorial theorems}
\begin{proof}[Proof of Theorem \ref{Theorem: Multidimensional Bog} via Theorem \ref{Theorem: Quant Poly Exp}]
Let $P=(P_1,\ldots,P_d):\Z^d \to \Z^d$ be as in the statement of the Theorem. Fix some $\delta>0$ and let $E\subset \Z^d$ have $d^*(E)\geq \delta$. Consider the product set $E':= E\times E \subset \Z^{2d}$. By Proposition \ref{Prop: FCP} there exists an ergodic action $T:\Z^{2d}\acts (X,\mu)$ and a set $A\subset X$ with $\mu(A)=d^*(E')\geq \delta^2$ satisfying
\begin{equation}\label{eq: 1}
\mu(A\cap T^{v}A) \leq d^*(E\cap (E+v)) \quad \text{for every } v\in \Z^{2d}.
\end{equation}
Define an auxiliary integer polynomial $Q:\Z^d \to \Z^{2d}$ by
\[Q(n)=(-P(n),n) \quad\text{for every }n\in \Z^d.\]
Our assumptions on $P$ ensure that the polynomial $Q:\Z^d \to \Z^{2d}$ has zero constant term and linearly independent component polynomials. We can then apply Theorem \ref{Theorem: Quant Poly Exp} to some $\varepsilon<\delta^2$, the system $(X\supset A,\mu)$ and the polynomial $Q$ to find some positive integer $k\leq k_0(Q,\delta,\varepsilon)$ and a $T^k$-ergodic component $\nu$ of $\mu$ with $\nu(A)\geq \mu(A)$ satisfying that
\begin{equation}
\nu\left( \bigcup_{n\in\Z^d}T^{Q(n)}A \right) > 1-\varepsilon.
\end{equation}
Fix any $m\in \Z^d$. Using that $\nu$ is invariant under the action of $k\Z^{2d}$ we also have that
\begin{equation}\label{eq: 2}
\nu\left( \bigcup_{n\in\Z^d}T^{Q(n)+(km,0)}A \right) > 1-\varepsilon.
\end{equation}
Since $\nu(A)\geq \mu(A)\geq \delta^2 > \varepsilon$ then the intersection of $A$ with the set measured in the left hand side of equation \eqref{eq: 2} has positive $\nu$-measure. This implies that there exists some $n\in \Z^d$ such that
\[\nu(T^{Q(n)+(km,0)}A \cap A)>0.\]
Of course $\nu$ is a $T^k$-ergodic component of $\mu$ so we also have that
\[\mu(T^{Q(n)+(km,0)}A \cap A)>0.\]
By equation \eqref{eq: 1} it follows that
\begin{equation*}
0<\mu(T^{Q(n)+(km,0)}A \cap A) \leq d^*\big(E' \cap \left(E' + Q(n)+(km,0)\right)\big),
\end{equation*}
which in particular establishes that $Q(n)+(km,0)\in E'-E'$, or equivalently the points
\begin{equation*}
x: = km-P(n)\quad \text{and}\quad  y:= n
\end{equation*}
are both in $E-E$.
Hence
\[  E-E + P(E-E) \ni x + P(y) = km.\]
Since $m$ was arbitrary the result follows.
\end{proof}
Theorem \ref{Theorem: Quantitative triangles} follows from the following dynamical consequence of Theorem \ref{Theorem: Quant Direct Exp} which is a quantitative strengthening of Theorem 1.4 in \cite{BjF}.
\begin{thm} \label{Theorem: Dynamical Prelim to Triangles}
For every integer $d\geq 2$ and every $\delta>0$ there exist positive integers $k_0 = k_0(\delta,d)$ and $m_0=m_0(\delta,d)$ such that the following holds. 
For every ergodic action $T:\Z^d \acts (X,\mu)$ and every set $A\subset X$ with $\mu(A)\geq \delta$ there exist some positive integers $k\leq k_0$, $m\leq m_0$ and a primitive vector $v\in \Z^d$ such that for every $v_1,\ldots,v_{d-1}\in \Z^d$ there exist $n_1,\ldots,n_{d-1}\in \Z$ with
\[\mu(A\cap T^{mv}A\cap T^{n_1v+ kv_1}A\cap \ldots \cap T^{n_{d-1}v+ k v_{d-1}}A)>0.\]
\end{thm}
\begin{proof}[Proof of Theorem \ref{Theorem: Dynamical Prelim to Triangles} via Theorem \ref{Theorem: Quant Direct Exp}] 
Let $T:\Z^d \acts (X,\mu)$ act ergodically and suppose $A\subset X$ has $\mu(A)\geq \delta$.
Set
\[\varepsilon: = \frac{\delta^2}{4d}\]
and apply Theorem \ref{Theorem: Quant Direct Exp} to obtain some positive integer $k\leq k_0(\mu(A),\varepsilon)$, a $T^k$-ergodic component $\nu$ of $\mu$ with $\nu(A)\geq \mu(A)$ and a primitive vector $v \in \Z^d$ such that
\begin{equation*}
\nu\left( \bigcup_{n\in \Z} T^{nv}A\right) > 1-\varepsilon.
\end{equation*}
We claim there exists some positive integer $m\leq \frac{2k}{\nu(A)}$ such that
\[\nu(A\cap T^{mv}A)>\frac{\nu(A)^2}{2}.\]
Indeed consider the sets $A,T^{kv}A,T^{2kv}A,\ldots,T^{k(M-1)v}A$, which all have $\nu$-measure equal to $\nu(A)$. If 
\[\nu(T^{ikv}A\cap T^{jkv})\leq \frac{\nu(A)^2}{2}\quad \text{for every }0\leq i<j\leq M-1,\]
then Jensen's inequality implies that
\begin{align*}
(M \nu(A))^2 = \left(\int \sum_{i=0}^{M-1} 1_{T^{ikv}A} \, d\nu \right)^2 &\leq \int \left(\sum_{i=0}^{M-1} 1_{T^{ikv}A}\right)^2 \, d\nu \\
&\leq M\nu(A) + \frac{M^2-M}{2} \nu(A)^2
\end{align*}
which is a contradiction if $M>\frac{2}{\nu(A)}$ say. Hence there exist some $0\leq i<j\leq \frac{2}{\nu(A)}$ such that
\[\nu(T^{ikv}A\cap T^{jkv}A) >\frac{\nu(A)^2}{2},\]
and since $k\Z^d$ preserves $\nu$ the claim follows with $m:=(j-i)k$.

For any $v_1,\ldots,v_{d-1} \in \Z^d$ set
\[ A_0:= A \cap T^{mv}A \quad \text{and} \quad A_i := \bigcup_{n\in \Z}T^{nv+kv_i}A \quad \text{for }i=1,\ldots,d-1.\]
Then $\nu(A_0)>\frac{\nu(A)^2}{2}$ and since $\nu$ is $k\Z^d$-invariant we also have that $\nu(A_i)>1-\varepsilon$ for $i=1,\ldots,d-1$. We then calculate
\begin{align*}
\nu(A_0\cap A_1\cap\ldots\cap A_{d-1}) &= 1-\nu(A_0^c \cup A_1^c \cup \ldots A_{d-1}^c)\\
&\geq 1- \sum_{i=0}^{d-1} \nu(A_i^c)\\
&\geq 1-\left((d-1)\varepsilon\right)-\left(1-\frac{\nu(A)^2}{2}\right)\\
&\geq \frac{\delta^2}{2}-(d-1)\varepsilon>0
\end{align*}
where the final inequality follows from our choice of $\varepsilon$. Using the definition of the $A_i$'s then we have shown that for any $v_1,\ldots ,v_{d-1}\in \Z^d$ there exist some $n_1,\ldots,n_{d-1}\in \Z$ such that
\begin{equation*}\label{eq: dynamical version of triangles}
\nu(A\cap T^{mv}\cap T^{n_1v+ kv_1}A\cap \ldots \cap T^{n_{d-1}v + k v_{d-1}}A)>0.
\end{equation*}
Since $\nu$ is a $T^k$-ergodic component of $\mu$ then the set measured in the above inequality also has positive $\mu$-measure. The theorem then follows with $m_0(\delta,d):= \frac{2}{\delta}k_0$.
\end{proof}
The following deduction of Theorem \ref{Theorem: Quantitative triangles} from Theorem \ref{Theorem: Dynamical Prelim to Triangles} is identical to the argument presented in \cite{BjF}, but we include it for completeness.
\begin{proof}[Proof of Theorem \ref{Theorem: Quantitative triangles} via Theorem \ref{Theorem: Dynamical Prelim to Triangles}]
Let $E\subset \Z^d$ have $d^*(E)>0$. By Proposition \ref{Prop: FCP} there exists an ergodic action $T:\Z^d \acts (X,\mu)$ and a set $A\subset X$ with $\mu(A)=d^*(E)$ satisfying
\begin{equation}\label{eq: FCP for triangles}
\mu \left(\bigcap_{v\in F} T^v A\right) \leq d^*\left(\bigcap_{v\in F} (E+v)\right) \quad \text{for every finite } F\subset \Z^d.
\end{equation}
If we combine equation \eqref{eq: FCP for triangles} with the conclusion of Theorem \ref{Theorem: Dynamical Prelim to Triangles} then we obtain positive integers $k\leq k_0(d^*(E),d)$, $m\leq m_0(d^*(E),d)$ and a primitive vector $v\in \Z^d$ such that for any $v_1,\ldots,v_{d-1}\in \Z^d$ there exist $n_1,\ldots,n_{d-1} \in \Z$ and some $v_0\in E$ such that
\begin{equation}\label{eq: vertex vectors in E}
v_0,\: \: v_0+mv, \: \: v_0+n_1v+kv_1,\: \:\ldots, \: \:v_0+n_{d-1}v+kv_{d-1} \in E.
\end{equation}
For any $d+1$ points $\lambda_0,\lambda_1,\ldots,\lambda_d \in \Z^d$ denote by $S(\lambda_0,\ldots,\lambda_d)$ the $d$-simplex with vertex set $\{\lambda_0,\ldots,\lambda_d\}$. That is $S(\lambda_0,\ldots,\lambda_d)$ is the convex hull of the points $\{\lambda_0,\ldots,\lambda_d\}$. The signed volume of a $d$-simplex $S(\lambda_0,\ldots,\lambda_d)$ can be calculated via the formula\footnote{For a proof of this fact see \cite{Stein}.}
\begin{equation*}
\mathrm{Vol}_d(S(\lambda_0,\ldots,\lambda_d))= \frac{\mathrm{det}(\lambda_1-\lambda_0,\lambda_2-\lambda_0,\ldots,\lambda_d-\lambda_0)}{d!}.
\end{equation*}
Hence if we denote by $\mathrm{VolSpec_d(E)}$ the set of all signed volumes of $d$-simplices whose vertex set is contained in $E$, then equation \eqref{eq: vertex vectors in E} implies that for any $v_1,\ldots,v_{d-1}\in \Z^d$ there exist $n_1,\ldots,n_{d-1} \in \Z$ so that
\begin{align}
&\frac{\mathrm{det}(mv,n_1v+kv_1,\ldots, n_{d-1}v+kv_{d-1})}{d!} \nonumber \\
&\qquad \qquad = mk^{d-1}\frac{\mathrm{det}(v,v_1,\ldots,v_{d-1})}{d!}\in \mathrm{VolSpec_d(E)}. \label{eq: area in volspeec}
\end{align}
It is known\footnote{See for instance Section II, Chapter 5 in \cite{Newman}.} that $v$ being primitive ensures there exists $v_1',v_2',\ldots,v_{d-1}' \in \Z^d$ for which
\[\mathrm{det}(v,v_1',\ldots,v_{d-1}')=1.\]
It follows that for any integer $l\in \Z$ we can pick
\[v_1 = lv_1', \: \: v_2 = v_2', \: \: \ldots \: \:, v_{d-1}=v_{d-1}'\]
in equation \eqref{eq: area in volspeec} to conclude that
\[l\frac{mk^{d-1}}{d!}\in \mathrm{VolSpec}_d(E).\]
Setting $K:= mk^{d-1}$ then the above readily implies that
\[ K\Z \subset \mathrm{VolSpec}_d(E).\]
Since $K\leq m_0 k_0^{d-1}$ then $K$ is bounded in terms of $\delta$ and $d$ as required.
\end{proof}

\section{\textbf{The measure increment argument}}\label{Section: Measure Increment}
Let $T: \Z^d \acts (X,\mu)$ be an ergodic action and let $A\subset X$ have $\mu(A)>0$. Bochner's theorem says that there exists a unique finite Borel measure $\sigma$ on $\T^d:= \R^d / \Z^d$ satisfying
\begin{equation}
\mu(A\cap T^{v}A)= \int_{\T^d} e(v\cdot\alpha)\, d\sigma(\alpha) \quad \text{for every }v\in \Z^d,
\end{equation}
where
\[e(x):=\exp(2\pi i x)\]
and $\cdot$ is the standard dot product.
We call $\sigma$ the spectral measure of $A$.
Any rational $\alpha \in \T^d$ can be uniquely written in the form
\[ \alpha = \left(\frac{p_1}{q_1},\ldots,\frac{p_d}{q_d}\right)\]
for integers $0\leq p_i < q_i$ with $\gcd(p_i,q_i)=1$ for each $i=1, \ldots, d$. For any rational $\alpha$ we use this form to define
\[\denom(\alpha):=\lcm(q_1,\ldots,q_d),\]
and for a positive integer $M$ we set
\[\mathrm{Rat}(M): = \{\text{rational } \alpha \in \T^d \setminus \{0\}\, : \, \denom(\alpha) \leq M\}. \]
Both proofs of Theorems \ref{Theorem: Quant Direct Exp} and \ref{Theorem: Quant Poly Exp} proceed by a measure increment argument which is a direct ergodic theoretic analogue of the now ubiquitous density increment argument first used by Roth in \cite{Roth}. This measure increment argument relies on two key observations. The first observation says the only obstruction to $A$ being sufficiently directionally or polynomially expandable is if $\sigma$ gives a large amount of mass to rationals with small denominator, i.e. if $\sigma(\Rat(M))$ is large for some integer $M>0$. The next two propositions formalise this observation.
\begin{pro}\label{Proposition: Small Rat(M) gives directional conclusion}
Let $\delta>0$ and $\varepsilon>0$. There exists a positive integer $M=M(\delta,\varepsilon)$ and a positive constant $\kappa = \kappa(\delta,\varepsilon)$ such that the following holds.
If $T:\Z^d \acts (X,\mu)$ is an ergodic action and $A\subset X$ has $\mu(A)\geq \delta$, spectral measure $\sigma$ and
\[\sigma(\Rat(M))< \kappa,\]
then there exists some primitive vector $v\in \Z^d$ for which
\[ \mu\left(\bigcup_{n\in\Z} T^{nv}A\right)> 1-\varepsilon.\]
\end{pro}
\begin{pro}\label{Proposition: Small Rat(M) gives polynomial conclusion}
Let $\delta>0$ and $\varepsilon>0$. Let $P=(P_1,\ldots,P_d): \Z^r \to \Z^d$ be an integer polynomial in $r$-variables with zero constant term such that the component polynomials are linearly independent. There exists a positive integer $M=M(\delta,\varepsilon,P)$ such that the following is true. If $T:\Z^d \acts (X,\mu)$ is an ergodic action and $A\subset X$ has $\mu(A)\geq \delta$, spectral measure $\sigma$ and
\[\sigma(\Rat(M))< \frac{\mu(A)^2\varepsilon^2}{4},\]
then
\[ \mu\left(\bigcup_{n\in\Z^r} T^{P(n)}A\right)> 1-\varepsilon.\]
\end{pro}
If the spectral measure does not give small mass to rationals with small denominator, then the second observation allows us to obtain a measure increment of $A$ with respect to a $T^k$-ergodic component. 
\begin{lem}\label{Lemma: Increment}
Let $T:\Z^d \acts (X,\mu)$ be an ergodic action and let $A\subset X$ have $\mu(A)>0$ and spectral measure $\sigma$. For any positive integer $M$ there exists a positive integer $k\leq M!$ and some $T^k$-ergodic component $\nu$ of $\mu$ such that
\[\nu(A)\geq \sqrt{\mu(A)^2 + \sigma(\Rat(M))}.\]
\end{lem}
\section{\bf{Proofs of Theorems \ref{Theorem: Quant Poly Exp} and \ref{Theorem: Quant Direct Exp}}}\label{Section: Proof of Quant Expan Theorems}
\begin{proof}[Proof of Theorem \ref{Theorem: Quant Poly Exp} via Proposition \ref{Proposition: Small Rat(M) gives polynomial conclusion} and Lemma \ref{Lemma: Increment}]
Let $T:\Z^d \acts (X,\mu)$ act ergodically and let $A\subset X$ have $\mu(A)>0$ and spectral measure $\sigma$. Fix $\varepsilon>0$ and let $P:\Z^r \to \Z^d$ be an integer polynomial with zero constant term and linearly independent components. Either the conclusion holds with $\nu = \mu$ or
\[ \mu \left( \bigcup_{n\in\Z^r}T^{P(n)}A\right)\leq 1-\varepsilon.\]
In the latter case Proposition \ref{Proposition: Small Rat(M) gives polynomial conclusion} ensures the existence of some positive integer $M_1$ such that
\[\sigma(\Rat(M_1))\geq \kappa_1=: \frac{\mu(A)^2\varepsilon^2}{4}.\]
By Lemma \ref{Lemma: Increment} then there exists some positive integer $k_1\leq M_1!$ and a $T^{k_1}$-ergodic component $\nu_1$ of $\mu$ such that
\[\nu_1(A) \geq \sqrt{\mu(A)^2 + \kappa_1} \geq \mu(A) + \frac{\kappa_1}{3}.\]
Either the conclusion holds with $k=k_1$ and $\nu=\nu_1$ or
\begin{equation}\label{eq: passing to next component poly}
1-\varepsilon \geq \nu_1 \left( \bigcup_{n\in\Z^r}T^{P(n)}A\right)\geq \nu_1\left( \bigcup_{n\in\Z^r}T^{P(k_1n)}A\right).
\end{equation}
Assume we are in the latter case. Since $P(0)=0$ then we can define another integer polynomial by 
$P^1(n):=P(k_1 n)/k_1$. Clearly $P^1(0)=0$. We claim that the components of $P^1$ are linearly independent. Indeed suppose some $a_1,\ldots,a_d \in \R$ have that
\[ 0=\sum_{i=1}^d a_i P^1_i(n)\quad \text{for all }n\in \Z^r.\]
Then by definition of $P^1$ we also have that
\[0= \sum_{i=1}^d a_i P_i(n) \quad \text{for all }n\in k_1\Z^r.\]
Only the zero polynomial can vanish on an entire lattice\footnote{See for example \cite{Alon}[Lemma 2.1].} and so we must conclude that
\[0= \sum_{i=1}^d a_i P_i,\]
which by linear independence of $P_1,\ldots,P_d$ implies that $a_1=\ldots=a_d=0$, proving the claim. If we denote the sub-action of $k_1\Z^d$ by $T_1$, that is
\[T_1^v = T^{k_1v} \quad \text{for all }v\in \Z^d,\]
then equation \eqref{eq: passing to next component poly} reads
\[ \nu_1 \left(\bigcup_{n\in\Z^r}T_1^{P^1(n)}A\right)  \leq 1-\varepsilon.\]
Since $T_1$ is ergodic with respect to $\nu_1$ then we can apply Proposition \ref{Proposition: Small Rat(M) gives polynomial conclusion} again to obtain some integer $M_2$ for which
\[\sigma_1(\Rat(M_2)) \geq \kappa_2 =: \frac{\nu_1(A)^2\varepsilon^2}{4}\]
where $\sigma_1$ is the spectral measure of $A$ with respect to $T_1: \Z^d \acts (X,\nu_1)$. By Lemma \ref{Lemma: Increment} there exists some positive integer $k_2 \leq M_2!$ and a $T_1^{k_2}$-ergodic component $\nu_2$ of $\nu_1$ with
\[\nu_2(A)\geq \nu_1(A) + \frac{\kappa_2}{3} \geq \mu(A) + \frac{\kappa_1}{3}+ \frac{\kappa_2}{3}.\]
It is easy to see that $\nu_2$ is a $T^{k_1 k_2}$-ergodic component of $\mu$, so either the conclusion holds with $k=k_1k_2$ and $\nu=\nu_2$ or
\[1-\varepsilon \geq \nu_2\left(\bigcup_{n\in \Z^r} T^{P(n)}A\right).\]
In the latter case we can then define $P^2(n):=\frac{P(k_1 k_2n)}{k_1k_2}$ and $(T_2^v)_{v\in \Z^d}:=(T^{k_1k_2 v})_{v\in \Z^d}$ to see that
\[ 1-\varepsilon \geq \nu_2\left(\bigcup_{n\in \Z^r} T_2^{P^2(n)}A\right)\]
and so on. If we set $\nu_0:=\mu$, then each $\kappa_i$ is of the form
\[ \kappa_i = \frac{\nu_{i-1}(A)^2\varepsilon^2}{4} \geq \frac{\mu(A)^2\varepsilon^2}{4}.\]
As $\nu_i(A)$ cannot exceed $1$ then this process must end in a finite number of steps $R$ bounded in terms of $\delta$ and $\varepsilon$. When the process terminates the conclusion of the theorem must hold with $k=k_1k_2\ldots k_R \leq M_1! M_2!\ldots M_R!$ and some $T^k$-ergodic component $\nu_R$ of $\mu$. Since each $M_i$ depended only on $\nu_i(A),\varepsilon$, and $P^i$, which in turn only depend on $\mu(A)$, $\varepsilon$ and $P$, then $k$ is bounded in terms of $\mu(A),\varepsilon$, and $P$ as claimed.
\end{proof}

\begin{proof}[Proof of Theorem \ref{Theorem: Quant Direct Exp} via Proposition \ref{Proposition: Small Rat(M) gives directional conclusion} and Lemma \ref{Lemma: Increment}]
Let $T:\Z^d \acts (X,\mu)$ act ergodically and let $A\subset X$ have $\mu(A)>0$ and spectral measure $\sigma$. Fix $\varepsilon>0$. Either the conclusion holds with $\nu = \mu$ or
\[ \mu \left( \bigcup_{n\in\Z}T^{nv}A\right) \leq 1-\varepsilon \quad \text{for every }v\in\Z^d.\]
In the latter case Proposition \ref{Proposition: Small Rat(M) gives directional conclusion} ensures the existence of some positive integer $M_1$ and some positive $\kappa = \kappa(\mu(A),\varepsilon)$ such that
\[\sigma(\Rat(M_1))\geq \kappa.\]
By Lemma \ref{Lemma: Increment} then there exists some positive integer $k_1\leq M_1!$ and a $T^{k_1}$-ergodic component $\nu_1$ of $\mu$ such that
\[\nu_1(A) \geq \sqrt{\mu(A)^2 + \kappa} \geq \mu(A) + \frac{\kappa}{3}.\]
Either the conclusion holds with $k=k_1$ and $\nu=\nu_1$ or
\[1-\varepsilon \geq \nu_1\left(\bigcup_{n\in\Z} T^{nv}A\right) \geq \nu_1\left(\bigcup_{n\in\Z} T^{nk_1v}A\right) =  \nu_1\left(\bigcup_{n\in\Z} T_1^{nv}A\right)\]
for all $v\in \Z^d$, where $(T_1^v)_{v\in \Z^d}:= (T^{k_1v})_{v\in \Z^d}$. Hence if we are in the latter case we can repeat the argument to obtain another mass increment for the set $A$ of size $\kappa/3$ with respect to some $T^{k_1k_2}$-ergodic component of $\mu$ and so on. All remaining details are as in the proof of Theorem \ref{Theorem: Quant Poly Exp}.
\end{proof}

\section{\textbf{$T^k$-ergodic components and eigenfunctions}}
\begin{defn}
Let $T:\Z^d \acts (X,\mu)$ act ergodically. For any $\alpha \in \T^d$, a function $f\in L^2(X,\mu)$ is called an $\alpha$-eigenfunction if
\[f\circ T^v = e(\alpha \cdot v) f \quad \text{for all } v\in \Z^d.\]
We denote the set of all $\alpha$-eigenfunctions by $\mathrm{Eig}_T(\alpha)$, and for any $R\subset \T^d$ we define
\[\Eig_T(R):= \overline{\mathrm{Span}\left\{f\in L^2(X,\mu)\,: \, f\in \Eig_T(\alpha)\text{ for some }\alpha \in R\right\}}.\]
\end{defn}
For $f,g\in L^2(X,\mu)$ we set
\[ \langle f,g\rangle := \int_X f \overline{g}\, d\mu.\]
It is not hard to see that $\Eig_T(\alpha)$ and $\Eig_T(\beta)$ are orthogonal whenever $\alpha \neq \beta \in \T^d$ and moreover ergodicity implies that each $\Eig_T(\alpha)$ has dimension at most $1$. Hence $\Eig_T(R)$ admits an orthonormal basis consisting $\alpha$-eigenfunctions, one for each $\alpha \in R$ whose eigenspace $\Eig_T(\alpha)$ is non-trivial.
\begin{lem}\label{Lemma: Spectral mass is L2}
Let $T:\Z^d \acts (X,\mu)$ be an ergodic action and let $A\subset X$ have $\mu(A)>0$ and spectral measure $\sigma$. For any $\alpha \in \T^d$ denote by $\mathcal{P}_{\Eig_T(\alpha)}$ the orthogonal projection onto $\Eig_T(\alpha)$. Then for any $\alpha \in \T^d$ we have that
\[\langle \mathcal{P}_{\Eig_T(\alpha)} 1_A,1_A \rangle = \sigma(\{\alpha\}),\]
and moreover $\mu(A)^2 = \sigma(\{0\})$.
\end{lem}
\begin{proof}
The mean ergodic theorem applied to the unitary action $\left( e(-\alpha \cdot v)T^v\right)_{v\in \Z^d}$ says that any $f\in L^2(X,\mu)$ satisfies
\[ \mathcal{P}_{\Eig_T(\alpha)} f = \lim_{N\to \infty}\frac{1}{|Q_N|}\sum_{v\in \Z^d} e(-v \cdot \alpha)T^v f.\]
So by continuity of the inner product and the dominated convergence theorem we can calculate
\begin{align*}
\langle \mathcal{P}_{\Eig_T(\alpha)} 1_A, 1_A \rangle &= \left\langle \lim_{N\to \infty}\frac{1}{|Q_N|}\sum_{v\in \Z^d} e(-v \cdot \alpha)T^v 1_A, 1_A \right\rangle \\
&= \lim_{N\to \infty}\frac{1}{|Q_N|}\sum_{v\in \Z^d} e(-v \cdot \alpha) \int_{\T^d} e(v\cdot \beta)\, d\sigma(\beta)\\
&=\int_{\T^d} \lim_{N\to \infty}\frac{1}{|Q_N|}\sum_{v\in \Z^d} e(v\cdot (\beta - \alpha)) \, d\sigma(\beta) \\
&= \int_{\T^d} 1_{\{\alpha-\beta = 0\}}\, d\sigma(\beta) = \sigma(\{\alpha\}).
\end{align*}
Notice that $\Eig_T(0)$ is exactly the space of $T$-invariant functions, so by ergodicity $\Eig_T(0)$ is the space of almost everywhere constant functions. Hence
\[ \langle \mathcal{P}_{\Eig_T(0)}1_A,1_A \rangle =\langle \mu(A) 1_X,1_A \rangle  = \mu(A)^2\]
as required.
\end{proof}
\begin{proof}[Proof of Lemma \ref{Lemma: Increment}]
Fix a positive integer $M$. Let 
\[R(M!)=\{\alpha \in \T^d\,:\, M!\alpha =0\in\T^d\}.\]
Pick an orthonormal basis for $\Eig_T(R(M!))$ consisting of one eigenfunction $f_\alpha \in \Eig_T(\alpha)$ for each $\alpha \in R(M!)$ such that $\Eig_T(\alpha)$ is non-trivial. In the case that some $\alpha \in R(M!)$ has $\Eig_T(\alpha) = \{0\}$, it will be convenient for notational purposes to let $f_\alpha = 0$. In any case then $\{f_\alpha\}_{\alpha \in R(M!)}$ is an orthonormal spanning set of $\Eig_T(R(M!))$. Since each $\Eig_T(\alpha)$ is at most $1$-dimensional then
\[\langle \mathcal{P}_{\Eig_T(\alpha)}1_A,1_A\rangle = |\langle 1_A,f_\alpha \rangle|^2 \quad \text{for every }\alpha \in R(M!),\]
and so Lemma \ref{Lemma: Spectral mass is L2} then implies that
\[\sigma(\Rat(M))+\sigma(\{0\}) =\sigma(\Rat(M))+\mu(A)^2  =  \sum_{\alpha \in \Rat(M)\cup\{0\}} |\langle 1_A,f_\alpha\rangle|^2.\]
Of course $\Rat(M)\cup\{0\} \subset R(M!)$ so we also have that
\begin{equation}\label{eq: big l2}
\sigma(\Rat(M))+\mu(A)^2 \leq  \sum_{\alpha \in R(M!)} |\langle 1_A,f_\alpha\rangle|^2.
\end{equation}
It is easy to see that
\[\Eig_T(R(M!)) \subset L^2(X,\mu)^{T^{M!}},\]
where $L^2(X,\mu)^{T^{M!}}$ is the space of all $M!\Z^d$-invariant functions in $L^2(X,\mu)$,
and so we can apply Parseval's formula to see that
\begin{align}
\sum_{\alpha \in R(M!)} |\langle 1_A,f_\alpha\rangle|^2 &= \int | \mathcal{P}_{\Eig_T(R(M!))}1_A|^2\, d\mu \nonumber \\
&\leq \int | \mathcal{P}_{L^2(X,\mu)^{T^{M!}}}1_A|^2\, d\mu.\label{eq: big L2}
\end{align}
By Proposition 
\ref{Prop: Tk ergodic components} there exists a finite number of $T^{M!}$-ergodic components $\nu_1,\ldots,\nu_n$ of $\mu$ for which
\[ \mu = \frac{1}{n}\sum_{i=1}^n \nu_i,\]
so combining equations \eqref{eq: big l2} and \eqref{eq: big L2} we see that
\begin{equation}\label{eq: big projections}
\sigma(\Rat(M))+\mu(A)^2 \leq \frac{1}{n}\sum_{i=1}^n \int | \mathcal{P}_{L^2(X,\mu)^{T^{M!}}}1_A|^2\, d\nu_i.
\end{equation}
Since each $\nu_i$ is $M!\Z^d$ ergodic then any $f\in L^2(X,\mu)^{T^{M!}}$ is constant $\nu_i$-almost everywhere for $i=1,\ldots,n$. It follows that
\[\int| \mathcal{P}_{L^2(X,\mu)^{T^{M!}}}1_A|^2 \, d\nu_i= \nu_i(A)^2 \quad \text{for each } i=1,\ldots,n,\]
and so equation \eqref{eq: big projections} reads
\[\sigma(\Rat(M))+\mu(A)^2 \leq \frac{1}{n}\sum_{i=1}^n \nu_i(A)^2.\]
The pigeonhole principle then yields some $i$ for which
\[ \nu_i(A) \geq \sqrt{\sigma(\Rat(M))+\mu(A)^2}.\]
\end{proof}
\section{\textbf{The polynomial dichotomy}}
During the proof of Proposition \ref{Proposition: Small Rat(M) gives polynomial conclusion} we will need to control polynomial exponential sums of the form
\begin{equation}\label{eq: generic poly exp sum}
\limsup_{N\to \infty} \labs{\frac{1}{N}\sum_{n=0}^{N-1} e\left( P(n)\cdot \alpha\right)}
\end{equation}
where $P: \Z \to \Z^d$ is an integer polynomial with linearly independent component polynomials and $\alpha \in \T^d$. Polynomial Weyl distribution implies that the expression in equation \eqref{eq: generic poly exp sum} is $0$ whenever $\alpha \not \in \Q^d / \Z^d$, indeed this is the content of Lemma \ref{Lemma: Can ignore irrationals}. It was observed in \cite{BuF} that a classical bound of Hua provides sufficient control of the expression in equation \eqref{eq: generic poly exp sum} in the case when $\alpha \in \T^d$ is rational, subject to the constraint that $P$ has bounded \textit{multiplicative complexity}.
\begin{defn}[Multiplicative complexity of polynomials]
An integer polynomial $P:\Z \to \Z^d$ has multiplicative complexity $Q$ if for all $a_1,\ldots,a_d,q\in \Z$ with $\gcd(a_1,\ldots,a_d,q)=1$ the polynomial
\[ \sum_{i=1}^D b_i n^i := (P(n)-P(0))\cdot (a_1,\ldots,a_d)\]
has that $\gcd(b_1,\ldots,b_D,q)\leq Q$.
\end{defn}
\begin{lem}[{\cite{BuF}[Proposition 2.2]}\footnote{The authors of \cite{BuF} provide more quantitative information about the nature of the function $\psi_P$, but the weaker formulation presented here suffices for our purposes.}]\label{Lemma: Hua}
Let $P:\Z \to \Z^d$ be an integer polynomial with bounded multiplicative complexity. Then there exists a decreasing function $\psi_P: \N \to [0,1]$ with $\lim_{q\to \infty}\psi_P(q) = 0$ such that every $\alpha \in \Q^d/ \Z^d$ with $\denom(\alpha)=q$ satisfies that
\[ \limsup_{N\to \infty} \labs{\frac{1}{N}\sum_{n=0}^{N-1} e\left( P(n)\cdot \alpha\right)} \leq \psi_P(q).\]
\end{lem}
\begin{lem}\label{Lemma: Independence implies bounded complexity}
If $P=(P_1,\ldots,P_d): \Z \to \Z^d$ is an integer polynomial with linearly independent component polynomials, then $P$ has bounded multiplicative complexity.
\end{lem}
\begin{proof}
It suffices to assume that $P(0)= 0$.
We must show there exists a constant $Q = Q(P)$ such that for any $a_1\ldots a_d, q\in \Z$ with $\gcd(a_1,\ldots,a_d,q)=1$ the polynomial
\[\sum_{j=1}^D b_j n^j := P(n)\cdot (a_1,\ldots,a_d),\]
where $D$ is the degree of $P$, satisfies that $\gcd(b_1,\dots,b_D,q)\leq Q$. Indeed let each
\[P_i(n) = c^i_1 n + \ldots +c^i_Dn^D\]
and let $B$ be the $D\times D$ matrix whose $i^\text{th}$ column is the coefficient vector $(c^i_1,\ldots,c^i_D)$ of $P_i$.
The coefficients $b =(b_1,\ldots, b_D)^\top$ are given by $B a$ where $a = (a_1,\ldots,a_d)^\top$. We can place $B$ into Smith normal form to obtain a decomposition $B=LDR$ for some $L \in \mathrm{SL}_D(\Z)$, $R \in \mathrm{SL}_d(\Z)$, and some diagonal matrix $D\in \mathrm{Mat}_{D\times d}(\Z)$ of the form $D= (D_1,\ldots,D_m,0,\ldots,0)$ for non-zero integers $D_1, D_2,\ldots D_m$ satisfying that $D_i$ divides $D_{i+1}$ for each $i=1,\ldots,m-1$. Since the components of $P$ are linearly independent then $\mathrm{rank}(B) = d$. It follows that
\[ d = \mathrm{rank}(B) \leq \min\{\mathrm{rank}(L),\mathrm{rank}(D),\mathrm{rank}(R)\},\]
and so $d \leq m$. It is easy to see that $\gcd(Ax) = \gcd(x)$ for all $x\in \Z^D$ and all $A\in \mathrm{SL}_D(\Z)$, hence
\begin{align*}
\gcd(Ba,q) &= \gcd(LDRa,q) \\
&= \gcd(DRa,q) \leq D_1 \gcd(Ra,q) = D_1\gcd(a,q) = D_1.
\end{align*}
\end{proof}

\begin{lem}[\cite{BuF}{[Lemma 4.3]}]\label{Lemma: Can ignore irrationals}
Let $P=(P_1,\ldots,P_d): \Z \to \Z^d$ be an integer polynomial with zero constant term such that $P_1,\ldots,P_d$ are linearly independent and let $T: \Z^d\acts (X,\mu)$ be a measure preserving action. Suppose that $f\in \Eig_{T}(\Q^d/\Z^d)^\perp$ i.e. $\langle f,f_{\alpha}\rangle = 0$ for all rational $\alpha \in \T^d$. Then
\[ \lim_{N\to \infty} \lnorm{\frac{1}{N}\sum_{n=0}^{N-1} T^{P(n)}f }_{L^2(X,\mu)}= 0.\]
\end{lem}
\begin{proof}[Proof of Proposition \ref{Proposition: Small Rat(M) gives polynomial conclusion}]
Fix $\delta,\varepsilon>0$ and let $P:\Z^r \to \Z^d$ be an integer polynomial in the $r$ variables $x_1,\ldots,x_r$. Suppose further that $P$ has zero constant term and that the component polynomials $P_1,\ldots,P_d$ are linearly independent. 
Let $\mathrm{deg}(P) = D$. We claim that the map sending $x_j \mapsto n^{{(D+1)}^j}$ is injective on the monomials appearing in the components of $P$. Indeed if $x_1^{i_1}\ldots x_r^{i_r}$ is a monomial appearing in some component of $P$ then we can calculate
\[ \prod_{j=1}^r x_j^{i_j} \mapsto \prod_{j=1}^r \left(n^{{(D+1)}^j}\right)^{i_j} = n^{\sum_{j=1}^r i_j (D+1)^j},
\]
and since each $i_j < (D+1)$ the claim then follows by the uniqueness of representations of integers in base $D+1$.
For each $i=1,\ldots, d$ we define a polynomial $Q_i\in \Z[n]$ via the formula
\[ Q_i(n): = P_i\left(n^{{(D+1)}},n^{{(D+1)}^2},\ldots,n^{{(D+1)}^r}\right)\]
and set
\[Q=(Q_1,\ldots,Q_d):\Z \to \Z^d.\]
Our claim ensures that each $Q_i$ has the same coefficients as $P_i$, and so our assumption that the components polynomials of $P$ are linearly independent implies that the component polynomials of $Q$ are also linearly independent. Lemma \ref{Lemma: Independence implies bounded complexity} then implies that $Q$ has bounded multiplicative complexity. We can then invoke Lemma \ref{Lemma: Hua} to obtain $\psi_{Q}$ as in the statement of the lemma, and let $M=M(\delta,\varepsilon,Q)$ be the smallest positive integer such that
\begin{equation}\label{eq: choice of M}
\psi_Q(q)< \frac{\delta\varepsilon}{2} \quad \text{for all } q>M.
\end{equation}
Set $\kappa:= \left(\frac{\delta\varepsilon}{2}\right)^2$. Let $T:\Z^d \acts (X,\mu)$ act ergodically and let $A\subset X$ have $\mu(A)\geq \delta$, spectral measure $\sigma$ and
\begin{equation}\label{eq: suppose ratM small}
\sigma(\Rat(M)) < \kappa.
\end{equation}
Suppose in order to derive a contradiction that the desired conclusion does not hold. Then there exists some $B\subset X$ with $\mu(B) \geq  \varepsilon$ satisfying
\begin{equation*}
0=\mu\left( \bigcup_{x\in \Z^r} T^{P(x)}A \cap B\right),
\end{equation*}
which in particular implies that
\begin{equation*}
0=\mu\left( \bigcup_{n\in \Z} T^{Q(n)}A \cap B\right).
\end{equation*}
It follows that $\langle T^{Q(n)} 1_A,1_B\rangle = 0$ for every $n\in \Z$ and so 
\[0=\lim_{N\to \infty} \Big\langle \frac{1}{N} \sum_{n=0}^{N-1}T^{Q(n)} 1_A,1_B\Big \rangle.\]
By Lemma \ref{Lemma: Can ignore irrationals} then
\begin{equation}\label{eq: inner prod is 0}
0 =  \lim_{N\to \infty} \Big\langle \frac{1}{N}\sum_{n=0}^{N-1} T^{Q(n)}\mathcal{P}_{\Eig_T(\Q^d / \Z^d)}1_A ,1_B \Big\rangle.
\end{equation}
%It is well known\footnote{See for instance the proof of Theorem 7.10 in \cite{EW}.} that $\mathcal{P}_{\Eig_T(\Q^d / \Z^d)}$ is a positive operator on $L^2(X,\mu)$ and so we must also have that
%\begin{equation}\label{eq: inner prod is 0}
%0= \Big\langle \frac{1}{N}\sum_{n=0}^{N-1} T^{P(n)}\mathcal{P}_{\Eig_T(\Q^d / \Z^d)}1_A ,1_B \Big\rangle \quad \text{for all }N.
%\end{equation}
By expanding $\mathcal{P}_{\Eig_T(\Q^d / \Z^d)}1_A$ into an orthonormal basis of rational eigenfunctions we can write
\[\mathcal{P}_{\Eig_T(\Q^d / \Z^d)}1_A = \sum_{q=1}^\infty \sum_{\denom(\alpha) = q} \langle 1_A,f_\alpha \rangle f_\alpha\]
where the second sum is over all rational $\alpha \in \T^d$ with $\denom(\alpha)=q$. Applying these observations to equation \eqref{eq: inner prod is 0} allows us to calculate
\begin{align*}
0 &= \lim_{N\to \infty}\Big\langle \frac{1}{N}\sum_{n=0}^{N-1} T^{Q(n)}\mathcal{P}_{\Eig_T(\Q^d / \Z^d)}1_A ,1_B \Big\rangle\\
&=\lim_{N\to \infty}\Big\langle \frac{1}{N}\sum_{n=0}^{N-1} \sum_{q=1}^\infty \sum_{\denom(\alpha) = q} \langle 1_A,f_\alpha \rangle T^{Q(n)}f_\alpha ,1_B \Big\rangle\\
&=\lim_{N\to \infty}\sum_{q=1}^\infty \sum_{\denom(\alpha) = q} \left(\frac{1}{N}\sum_{n=0}^{N-1}  e(Q(n)\cdot \alpha)\right) \langle  f_\alpha ,1_B \rangle \langle 1_A,f_\alpha \rangle.
\end{align*}
For each $N\in \N$ denote
\[ T_N(q) :=\sum_{\denom(\alpha) = q}\left(\frac{1}{N}\sum_{n=0}^{N-1}  e(Q(n)\cdot \alpha)\right) \langle  f_\alpha ,1_B \rangle \langle 1_A,f_\alpha \rangle.\]
The only rational in $\T^d$ with denominator $1$ is $0$ so
\[T_N(1)=\mu(A)\mu(B)>\delta \varepsilon \quad \text{for every } N \in \N.\]
We can then re-write the last line of our calculation as
\begin{equation} \label{eq: splitting into three}
0 = \mu(A)\mu(B) + \lim_{N\to \infty} \left[\sum_{1<q\leq M} T_N(q) + \sum_{q>M}T_N(q)\right].
\end{equation}

We will show that the later two terms in equation \eqref{eq: splitting into three} are small enough to ensure that equation \eqref{eq: splitting into three} is in fact a contradiction, which will finish the proof. More precisely we claim that
\[ \limsup_{N\to \infty}\labs{\sum_{1<q\leq M} T_N(q)} < \frac{\delta\varepsilon}{2} \quad \text{and} \quad \limsup_{N\to \infty}\labs{\sum_{q> M} T_N(q)} < \frac{\delta\varepsilon}{2}.\]
Let us deal with the small denominators first. Recall from our proof of Lemma \ref{Lemma: Increment} that
\begin{equation}\label{eq: formula for spectral mass}
|\langle 1_A,f_\alpha \rangle|^2 = \sigma(\{\alpha\}).
\end{equation}
Now using the triangle inequality, the trivial bound on the exponential sum, Cauchy Schwarz, the Bessel inequality, equation \eqref{eq: formula for spectral mass} and equation \eqref{eq: suppose ratM small} we can estimate
\begin{align}
\limsup_{N\to \infty}\labs{\sum_{1<q\leq M} T_N(q)} &\leq  \sum_{\alpha \in \Rat(M)}  |\langle  f_\alpha ,1_B \rangle| |\langle 1_A,f_\alpha \rangle| \nonumber \\
& \leq \left(\sum_{\alpha \in \Rat(M)}\labs{\langle 1_A,f_\alpha\rangle}^2 \sum_{\alpha \in \Rat(M)} \labs{\langle f_\alpha,1_B\rangle}^2\right)^{1/2}\nonumber \\
&\leq \sqrt{\sigma(\Rat(M))}\nonumber \\
&<\sqrt{\kappa}=\frac{\delta\varepsilon}{2}. \label{eq: bound on small denoms}
\end{align}
where in the final equality we have used our choice of $\kappa$. For the large denominators we again use the triangle inequality, Cauchy Schwarz and the Bessel inequality, however instead of using the trivial bound for the exponential sum we use equation \eqref{eq: choice of M}. Indeed letting $C:= \Q^d / \Z^d \setminus \Rat(M)$ we have that
\begin{align*}
\limsup_{N\to \infty}\labs{\sum_{q>M} T_N(q)} &\leq \sup_{ \alpha \in C} \limsup_{N\to \infty}\labs{\frac{1}{N}\sum_{n=0}^{N-1}  e(Q(n)\cdot \alpha)} \\
& \qquad \times \left(\sum_{\alpha \in C} \labs{\langle 1_A,f_\alpha\rangle}^2 \sum_{\alpha \in C} \labs{\langle f_\alpha,1_B\rangle}^2\right)^{1/2}\\
&\leq \psi_Q(M) < \frac{\delta \varepsilon}{2}.
\end{align*}
\end{proof}

\section{\textbf{The directional dichotomy}}
Given any $v\in \Z^d$ let $L_v^\perp \subset \T^d$ denote the annihilator of $v$ inside $\T^d$, i.e.
\[L_v^\perp:=\{\alpha \in \T^d \, : \, v\cdot \alpha =0 \in \T^d\}.\]
The following lemma is an important observation from \cite{BjF} that reduces the study of expansive directions for ergodic $\Z^d$-systems to the study of annihilators inside $\T^d$.
\begin{lem}[\cite{BjF}{[Lemma 3.2]}]\label{Lemma: Directional orbit bounded from below by spectral mass of line}
Let $T:\Z^d \acts (X,\mu)$ act ergodically and suppose $A\subset X$ has $\mu(A)>0$ and spectral measure $\sigma$. For any $v\in \Z^d$ we have that
\[\mu \left( \bigcup_{n\in\Z} T^{nv} A\right) \geq \frac{\mu(A)^2}{\sigma(L_v^\perp)}.\]
\end{lem}
\begin{defn}[Haystacks]
Let $\mathcal{P}$ denote the set of all primitive vectors in $\Z^d$, i.e. those for which the $\gcd$ of it's components is $1$. An infinite set $H\subset \mathcal{P}$ is called a haystack if any distinct $v_1,\ldots,v_d \in H$ are linearly independent.
\end{defn}
There are many different ways to construct haystacks in $\Z^d$, see for instance Lemma 2.4 in \cite{BjF}. Let us fix some haystack $H$ for the remainder of the section. For each positive integer $M$ define
\begin{equation}\label{eq: Def of bounded part of haystack}
H_M: = \left\{v\in H\, : \, \norm{v}_\infty \leq \left(\frac{M}{d!}\right)^{\frac{1}{d}}\right\} \nonumber
\end{equation}
where $\norm{v}_\infty$ denotes the largest absolute value of the components of $v$.
\begin{lem} \label{Lemma: L perp in RatM}
Any distinct $v_1,\ldots, v_d\in H_M$ have that
\[\cap_{i=1}^d L_{v_i}^\perp \subset \mathrm{Rat}(M).\]
\end{lem}
\begin{proof}
Pick any distinct $v_1,\ldots,v_d\in H_M$ and let $A$ be the matrix whose rows are $v_1,\ldots,v_d$. Then
\[0<|\mathrm{det}(A)|\leq M.\]
By definition if $\alpha \in \cap_{i=1}^d L_{v_i}^\perp$ then there exist some $w_\alpha \in \Z^d$ such that $A \alpha = w_\alpha$, and since $A$ is invertible then $\alpha = A^{-1}w_\alpha$. The entries of $A^{-1}$ are all rational numbers inside $\Z / \det(A)$, and so each component of $\alpha$ is also a rational number inside $\Z / \det(A)$ which ensures that $\alpha \in \mathrm{Rat}(M)$.
\end{proof}
\begin{lem}\label{Lemma: Line with small irrational mass}
For any $\gamma>0$ there exists a positive integer $M=M(\gamma)$ such that the following is true. For any finite Borel measure $\sigma$ on $\T^d$ with $\sigma(\T^d)\leq 1$ and
\[\sigma(\Rat(M))\leq\frac{\gamma}{2}\]
there exists a vector $v\in H_M$ with
\[\sigma(L_v^\perp \setminus \{0\}) \leq \gamma.\]
\end{lem}
\begin{proof}
Let $M$ be a positive integer to be later specified.
For any $v\in \Z^d$ let $L_v:= L_v^\perp \setminus (\Rat(M) \cup \{0\})$. Lemma \ref{Lemma: L perp in RatM} and the definition of $L_v$ together imply that $\sum_{v\in H_M} 1_{L_v} \leq d-1$. We can then estimate
\[\labs{H_M} \min_{v \in H_M} \sigma(L_v) \leq \sum_{v \in H_M} \sigma(L_v) = \int_{\T^d} \sum_{v \in H_M} 1_{L_v} \, d\sigma \leq (d-1),\]
or equivalently
\[ \min_{v \in H_M} \sigma(L_v) \leq \frac{(d-1)}{\labs{H_M}}.\]
If we pick $M= M(\gamma)$ large enough so that the right hand side of the above equation is at most $\gamma/2$ then there must be some $v\in H_M$ for which $\sigma(L_v)\leq \gamma/2$. By the hypothesis $\sigma(\Rat(M))\leq \frac{\gamma}{2}$ and so it follows that 
\[ \sigma(L_v^\perp\setminus \{0\}) \leq \sigma(L_v) + \sigma(\Rat(M)) \leq \gamma\]
as required.
\end{proof}
\begin{proof}[Proof of Proposition \ref{Proposition: Small Rat(M) gives directional conclusion}]
Fix $\delta>0$ and $\varepsilon>0$. There exists a positive constant $\gamma = \gamma(\delta,\varepsilon)$ such that 
\[\frac{\delta^2}{\delta^2+\gamma} > 1-\varepsilon.\]
Let $T:\Z^d \acts (X,\mu)$ act ergodically and suppose $A\subset X$ has $\mu(A)\geq \delta$ and spectral measure $\sigma$. Our choice of $\gamma$ and the fact that $\mu(A)\geq \delta$ together ensure that 
\[\frac{\mu(A)^2}{\mu(A)^2+\gamma} > 1-\varepsilon.\]
By Lemma \ref{Lemma: Line with small irrational mass} there exists a positive integer $M=M(\gamma)$ such that if $\sigma(\Rat(M))\leq \gamma/2$ then there exists some $v\in \Z^d$ with $\sigma(L_v^\perp \setminus \{0\}) \leq \gamma$. Recall from Lemma \ref{Lemma: Spectral mass is L2} that $\sigma(\{0\})=\mu(A)^2$, and so $\sigma(L_v^\perp) \leq  \gamma + \mu(A)^2$. We can then use Lemma \ref{Lemma: Directional orbit bounded from below by spectral mass of line} to see that
\[\mu\left(\bigcup_{n\in\Z} T^{nv}A\right) \geq \frac{\mu(A)^2}{\sigma(L_v^\perp)} \geq \frac{\mu(A)^2}{\mu(A)^2+\gamma} >  1-\varepsilon\]
as required.

\end{proof}
\section{\textbf{Failure of full expansivity}}\label{Section: Counter examples}
The system in the following example is attributed to Bergelson and Ward and was used by the authors of \cite{RRS}[Example 2.11] as an example of a system with no ergodic directions.\footnote{See Definition 1.5 in \cite{BjF}.} We observe that the same system can be used to show that the $\varepsilon =0$ version of Theorem \ref{Theorem: Quant Direct Exp} fails to hold. In other words, we construct an ergodic action $T:\Z^d \acts (X,\mu)$ with a positive measure set $A\subset X$ such that for every $k$, every $T^k$-ergodic component $\nu$ of $\mu$ satisfies that
\[ \nu \left( \bigcup_{n\in \Z} T^{nv}A\right)<1 \quad \text{for every } v\in \Z^d.\]
\begin{exam}[Failure of full directional expansion]
Let $S: \Z \acts (Y,\nu)$ be a weak mixing system and equip $X:=\prod_{i \in \N} Y$ with the product measure $\mu:=  \bigotimes_{i\in \N} \mu_i$. Let $(\eta_i)_{i\in \N}$ be a fixed enumeration of $\Z^d \setminus \{0\}$ and define a $\Z^d$ action $T$ on $(X,\mu)$ by
\[T^v (x_i)_{i\in \N}:= (S^{v\cdot \eta_i} x_i)_{i\in \N} \quad \text{for } v\in \Z^d \text{ and } (x_i)_{i\in \N} \in X,\]
where $\cdot$ is the standard dot product on $\Z^d$. It can be checked\footnote{See for example \cite{RRS}[Proposition 2.9, Example 2.10 and Example 2.11].} that $T:\Z^d \acts (X,\mu)$ is weak mixing and hence totally ergodic. Total ergodicity ensures that for every integer $k$, the only $T^k$-ergodic component of $\mu$ is $\mu$ itself, and so it suffices to construct a positive measure set $A \subset X$ so that
\[ \mu\left( \bigcup_{n\in\Z} T^{nv} A \right) <1 \quad \text{for every }v\in \Z^d.\]
For any $v \in \Z^d$ there clearly exists some $i_v \in \N$ with $v\cdot \eta_{i_v} = 0$. The subgroup $\Z v$ then acts trivially on the $i_v^{\text{th}}$ component of $X$.
Pick some sequence of sets $A_i \subset Y$, such that $1-1/i^2 \leq \nu(A_i)<1$, and define
$A := \bigcap_i \pi_i^{-1}A_i$.
Then $0<\mu(A) < 1$, but for any vector $v\in \Z^d$,
\[\bigcup_{n\in \Z} T^{nv} A \subset \pi_{i_v}^{-1}A_{i_v}, \]
and the right hand side has $\mu$ measure equal to $\nu(A_{i_v})<1$.
\end{exam}
Next we turn out attention to task of showing that the $\varepsilon=0$ version of Theorem \ref{Theorem: Quant Poly Exp} fails to hold. We only consider the case $r=1$.
\begin{lem}\label{Lemma: infinitely many primes with bounded image}
For any polynomial $P \in \Z[n]$ with $\deg (P) \geq 2$ there exists a constant $\lambda\in (0,1)$ and infinitely many primes $p$ for which the set
\[ V(P,p):= \{ P(n) \mod p \, : \, n\in \Z\}\]
satisfies
\[ \labs{V(P,p)} \leq \lambda p.\]
\end{lem}
\begin{proof}
It is known\footnote{See for instance Case A on page 1 of \cite{Fried}.} that since $P$ is non-linear then there exist infinitely many primes for which $\labs{V(P,p)}< p$.  Proposition 2.11 (a) in \cite{Turnwald} states that
\[\labs{V(P,p)}< p \implies \labs{V(P,p)}\leq \left(1-\frac{1}{2\deg P}\right)p\]
and so the result follows.
\end{proof}
We remark that in the case when $P(n)=n^2$ then conclusion of Lemma \ref{Lemma: infinitely many primes with bounded image} can be seen more directly from the well known fact that there are only $(p+1)/2$ squares mod $p$ for any odd prime $p$, so in this case we can take $\lambda = 2/3$ say and the conclusion holds for all primes larger than $2$.
\begin{pro}\label{Proposition: No full poly exp}
For any integer polynomial $P =(P_1,\ldots,P_d) : \Z \to \Z^d$ of degree at least $2$ there exists an ergodic action $T : \Z^d \acts (X,\mu)$ and a positive measure set $A$ such that for every $k$, every $T^k$-ergodic component $\nu$ of $\mu$ satisfies that
\[ \nu \left(\bigcup_{n\in \Z} T^{P(n)}A\right) < 1.\]
\end{pro}
\begin{proof}
We first note that it suffices to prove the case $d=1$. Indeed, assume the $d=1$ case has been shown. Since $P$ has degree at least $2$ then there must be some $j\in \{1,\ldots,d\}$ such that $\deg P_j \geq 2$. Let $T: \Z \acts (X,\mu)$ and $A\subset X$ be as in the conclusion  of the $d=1$ case of the proposition applied to $P_j$. We can extend $T$ to an ergodic $\Z^d$ action on $(X,\mu)$ by letting any vector $(v_1,\ldots,v_d)\in \Z^d$ act by $T^{v_j}$ and the result follows.

For the remainder of the proof we take $d=1$, so that $P \in \Z[x]$ has $\deg P \geq 2$.
By Lemma \ref{Lemma: infinitely many primes with bounded image} there exists some $\lambda \in (0,1)$ and an increasing sequence of primes $\{p_i\}_{i=1}^\infty$ so that
\begin{equation}\label{eq: bound on images mod p}
\labs{V(P,p)}\leq \lambda p.
\end{equation}
Define the sequence $q_i$ by 
\[q_1=p_1,\quad q_2 = p_2\times p_3,\quad  q_3 = p_4\times p_5 \times p_6 \quad \text{ and so on.}\]
That is
\[q_1:= p_1 \quad \text{and } \quad q_i := \prod_{k=1}^i p_{1+2+\ldots +(i-1)+k} \quad \text{for } i\geq 2.\]
For each $i$ equip the space $X_i = \Z / q_i \Z$ with the counting probability measure $\mu_i$ and equip the product space $X = \prod_{i=1}^\infty X_i$ with the product measure $\mu = \bigotimes_{i=1}^\infty \mu_i$.
Define an action $T: \Z \acts (X,\mu)$ by
\[T^n(x_i)_{i=1}^\infty = (x_i+n\: (\mathrm{mod}\: q_i))_{i=1}^\infty \text{ for each } n\in \Z \text{ and } (x_i)_{i=1}^\infty \in X.\]
For each positive integer $i$ denote by $T_i$ the induced map on $X_i$. Notice that our system is a group rotation because $Tx = x+a$ where $a=(1+q_1\Z,1+q_2\Z,\dots)\in X$. The fact that $\mathrm{gcd}(q_i,q_j)=1$ for all $i\neq j$ together with the Chinese remainder theorem imply that the subgroup $\{na\}_{n\in \Z}$ is dense in $X$, and so $T$ is ergodic by Theorem 4.14 in \cite{EW}.

For a positive integer $q$ consider the set
\[S(q) = \{ P(n) \mod q \, : \, n\in \Z\}.\]
If $q$ is square free with prime factorisation $q = r_1\times \ldots \times r_m$ then the Chinese remainder theorem tells us that $S(q)$ is in bijection with set of $m$ tuples $(a_1,\ldots,a_m) \in V(P,r_1)\times \ldots \times V(P,r_m)$. For any positive integer $i$ we can then use equation \eqref{eq: bound on images mod p} to estimate
\[ \mu_i(S(q_i))  = \frac{1}{q_i} \times \labs{ S(q_i)} = \frac{1}{q_i}\prod_{k=1}^i \labs{V(P,p_{1+2+\ldots +(i-1)+k})} \leq \lambda ^i.\]
For each positive integer $i$ let 
\[A_i =  X_i \setminus (-S(q_i))\]
where $-S(q_i) = \{ - s \, : \, s \in S(q_i)\}$. We define $A = \prod_{i=1}^\infty A_i$, the point being that
\begin{equation}\label{eq: Poly orbit of A lacking 0}
0 +q_i\Z \not \in \bigcup_{n\in \Z} T_i^{P(n)}A_i \quad \text{for every positive integer } i.
\end{equation}
Since $\lambda \in (0,1)$ then
\[1-\lambda^i \leq \mu_i(A_i) \leq \mu_i\left(\bigcup_{n\in \Z} T_i^{P(n)} A_i\right) <1 \quad \text{for every positive integer } i\]
and so by the convergence properties of infinite products
\[0<\mu(A) \leq \mu\left(\bigcup_{n\in \Z} T^{P(n)} A\right) <1.\]
It remains to show that for every positive integer $k$, every $T^k$-ergodic component $\nu$ of $\mu$ has that
\[\nu \left(\bigcup_{n\in\Z} T^{P(n)} A\right)<1.\]
Fix some positive integer $k$ and consider the finite set
\[ J:= \{ j  \, : \, \gcd(q_j,k)>1\}.\]
Our space $X$ factors into $X_J:= \prod_{j\in J}X_j$ and $X' = X / X_J$ via the obvious factor maps $\pi_J:X\to X_J$ and $\pi_{J'}: X \to X'$. Let $T_J:= \pi_J\circ T$ and $T_{J'}:= \pi_{J'} \circ T$ be the induced $\Z$ actions. We claim that up to the measure $\mu$, all positive measure $T^k$-invariant sets are of the form $\pi_J^{-1}D$ for some $D\subset X_J$. Indeed let $C\subset X$ be a positive measure $T^k$-invariant set and write
\[C = \bigcup_{y \in \pi_J(C)} \{y\}\times C_y\]
where $C_y\subset X_{J'}$ for each $y\in \pi_J(C)$. Let $M=k\prod_{j\in J}q_j$ and notice that $T^M_J$ acts trivially on $X_J$. The $T^k$-invariance of $C$ then allows us to calculate
\begin{align*}
\bigcup_{y \in \pi_J(C)} \{y\}\times C_y &=T^{M}\left(\bigcup_{y \in \pi_J(C)} \{y\}\times C_y\right) \\
&=\bigcup_{y \in \pi_J(C)}T_{J}^M\{y\} \times T^M_{J'} C_y\\
&=\bigcup_{y \in \pi_J(C)}\{y\} \times T^M_{J'} C_y,
\end{align*}
which implies that $C_y$ is $T^M_{J'}$-invariant for every $y \in \pi_J(C)$. On the other hand $\gcd(M,q_i)=1$ for every $i\not \in J$ and so the same argument used to show that $T$ is ergodic also implies that $T^M_{J'}$ is ergodic with respect to the measure $\mu_{J'}:= \pi_{J'}^* \mu$. It follows that for each $y\in \pi_J(C)$, we must have that $\mu_{J'}(C_y) \in \{0,1\}$. Since we only care about the form of $C$ up to $\mu$ then we can ignore those $y$'s for which $\mu_{J'}(C_y)=0$, and all remaining $y$'s in $\pi_J(C)$ will have $C_y = X_{J'}$ up to $\mu_{J'}$, which proves the claim.
Any $T^k$-ergodic component of $\mu$ is of the form $\mu(\cdot \,| \,C)$ for some positive measure $T^k$-invariant set $C\subset X$. By the claim we can assume $C=\pi_J^{-1}D$ for some $D\subset X_J$, and so for all $i \not \in J$ we have that $C$ contains the positive $\mu$-measure set
\[C_0(i): = \pi_{J}^{-1}D \cap \pi_{i}^{-1}\{0\}.\]
On the other hand equation \eqref{eq: Poly orbit of A lacking 0} implies that
\[C_0(i) \cap  \bigcup_{n\in \Z} T^{P(n)}A = \emptyset\]
for every positive integer $i$, and so we must have that 
\[\mu \left( \bigcup_{n\in \Z} T^{P(n)}A \, \Bigg| \, C \right) < 1\]
as required.
\end{proof}
\section{\textbf{A counter example to the pinned version of the polynomial Bogolyubov theorem}} \label{Section: Counter example to Bog}

\begin{lem}\label{Lemma: No multiple recurrence for A}
Let $P\in \Z[n]$ have $P(0)=0$ and $\deg P\geq 2$. Let $T:\Z\acts (X,\mu)$ and $A\subset X$ be as in the proof of the of Proposition \ref{Proposition: No full poly exp} applied to the polynomial $P$. For every positive integer $k$ there exists some positive integer $m$ such that for every $l_1,\ldots,l_m \in \Z$ we have that
\[\mu\left(A\cap \bigcap_{n=1}^m T^{P(l_n)-kn} A\right)=0.\]
\end{lem}
\begin{proof}
Let
\[A_i' = \bigcup_{n\in \Z} T_i^{P(n)} A_i \quad \text{and} \quad A' = \bigcup_{n\in \Z} T^{P(n)} A \subseteq \prod_{i=1}^\infty A_i'.\]
Let $k$ be a positive integer. For any positive integers $m$ and $i$, and any $l_1,\ldots,l_m \in \Z$ we have that
\[\mu\left(A\cap \bigcap_{n=1}^m T^{P(l_n) - kn}A\right) \leq \mu\left(A' \cap \bigcap_{n=1}^m T^{-kn} A'\right) \leq \mu_i\left(\bigcap_{n=1}^m T_i^{-kn}A_i'\right).\]
Suppose in order to derive a contradiction that for each positive integer $m$ there existed $l_1,\ldots,l_m \in \Z$ for which the left hand side of the above equation was positive. Then for every $i$, the set $A_i'$ would admit arbitrarily long arithmetic progressions with common difference $k$. There exist some (of course many) $i$'s for which $\mathrm{gcd}(k,q_i)=1$, which in particular ensures that the multiples of $k$ generate all of $\Z/ q_i \Z$. It follows that for these $i$'s, the set $A_i'$ can only have arbitrarily long arithmetic progressions of common difference $k$ if $A_i' = \Z / q_i \Z$, but by construction every $i$ satisfies that $\mu_i(A_i')<1$, hence this cannot be.
\end{proof}
\begin{proof}[Proof of Theorem \ref{Theorem: Counter-example to pinned Bog}]
For a set $F\subset \Z^2$ define 
\[\Delta(F): = \{ x+P(y) \, : \, (x,y)\in F\}. \]
We must show that there exists some $E\subset \Z$ with $d^*(E)>0$ satisfying that for every positive integer $k$, there exists some positive integer $m$ such that
\[ \{k,2k,\ldots,mk\} \not \subset \Delta((E-a)\times (E -b)) \quad \text{for every } a,b\in E.\]
So let $T:\Z \acts (X,\mu)$ and $A\subset X$ be as in Proposition \ref{Proposition: No full poly exp} applied to the polynomial $P$. Using Lemma \ref{Lemma: No multiple recurrence for A} and the pointwise ergodic theorem, for almost every point $x\in X$, the set of return times of $x$ to $A$,
\[E_x:=\{n\in \Z \, : \, T^nx\in A\}\]
satisfies the following property.
For every positive integer $k$ there exists some positive integer $m$ for which
\[ E_x \cap (E_x+(k-P(l_1)))\cap \ldots \cap (E_x+(km-P(l_m))) = \emptyset\]
for every $l_1,\ldots,l_m \in \Z$. This implies that
\[ \{k,2k,\ldots,mk\} \not \subset \Delta((E_x-a) \times \Z) \quad \text{for every } a\in E_x.\]
Clearly however we have that
\[\Delta((E_x-a) \times \Z) \supset \Delta((E_x-a)\times(E_x-b))\quad \text{for every } a,b\in E_x\]
and so the result follows.
\end{proof}
\section{\textbf{Appendix}}\label{Section: Appendix}
\begin{pro}
The requirement in Theorem \ref{Theorem: Multidimensional Bog} that no non-trivial linear combination of the components of $P$ be a linear polynomial is necessary.
\end{pro}
\begin{proof}
Let $P=(P_1,\ldots,P_d) : \Z^d \to \Z^d$ be an integer polynomial with zero constant term and suppose that there exists $\alpha_1,\ldots,\alpha_d\in \Z$ not all zero satisfying that
\[ \sum_{i=1}^d \alpha_i P_i(x_1,\ldots,x_d) = \sum_{i=1}^d \beta_i x_i\]
for some $\beta_1,\ldots,\beta_d \in \Z$. Consider the product set
\[E:= B(\alpha,\varepsilon) \times \ldots \times B(\alpha,\varepsilon) \subset \Z^d\]
where $B(\alpha,\varepsilon)$ is the Bohr set
\[ B(\alpha,\varepsilon):= \{n \in \Z \, : \, n\alpha \in (-\varepsilon,\varepsilon) \: (\mathrm{mod} \,1)\}\]
for some irrational $\alpha$ and some small $\varepsilon>0$.
Since $d^*(E)>0$ then if the theorem holds for the polynomial $P$ there must be some positive integer $k$ such for every $m= (m_1,\ldots,m_d) \in \Z^d$  there exist $x=(x_1,\ldots,x_d)$ and $y=(y_1,\ldots,y_d)$ both in $E-E$ with
\[ km_i = x_i+P_i(y)\quad \text{for each }i=1,\ldots , d.\]
The above equations can be rearranged to read
\[ k\sum_{i=1}^d \alpha_i m_i = \sum_{i=1}^d \alpha_i (P_i(y) + x_i) = \sum_{i=1}^d \left(\beta_i y_i + \alpha_i x_i\right),\]
which in particular implies that
\[ \sum_{i=1}^d  \beta_i(B(\alpha,\varepsilon)-B(\alpha,\varepsilon)) + \alpha_i (B(\alpha,\varepsilon)-B(\alpha,\varepsilon))\]
contains a subgroup. On the other hand the triangle inequality implies that the above set is contained inside the Bohr set
\[B\left(\alpha,2\varepsilon\sum_{i=1}^d \left(\labs{\alpha_i} + \labs{\beta_i}\right)\right)\]
and so cannot contain a subgroup provided that $\varepsilon$ is sufficiently small.
\end{proof}
The following argument is identical to the one presented in \cite{Bu}[Proposition A.2], however we have chosen to include it for the sake of completeness.
\begin{proof}[Proof of Proposition \ref{Prop: Tk ergodic components}]
Let $T: \Z^d \acts(X,\mu)$ be ergodic and consider the collection
\[ \mathcal{C}:= \{C\subset X\, : \, \mu(C)>0 \text{ and } k\Z^d C \subset C\}.\]
Set
\[\kappa:= \inf_{C \in \mathcal{C}} \mu(C).\]
We first show that $\kappa \geq 1/k^d$.
Indeed pick coset representatives $v_1,\ldots, v_{k^d}$ for $k\Z^d$ inside $\Z^d$ and let $C\in \mathcal{C}$. Then the set
\[A:= \bigcup_{i=1}^{k^d} T^{v_i}C\]
is invariant under all of $\Z^d$ so ergodicity implies that $\mu(A)=1$. Since $T$ preserves $\mu$ then
\[1=\mu(A) \leq \sum_{i=1}^{k^d} \mu(T^{v_i}C) = k^d \mu(C)\]
as required.
By definition of $\kappa$ there exists some $C\in \mathcal{C}$ with
\[\kappa \leq \mu(C)< \kappa + \kappa/2.\]
We claim that $k\Z^d$ acts ergodically on $C$, so that actually $\mu(C) = \kappa$. Indeed, if the claim fails then there exists some $C'\subset C$ with $k\Z^dC' \subset C'$ and $\mu(C')\in (0,\mu(C))$. This implies that $C' \in \mathcal{C}$ and so $\mu(C')\in [\kappa,\mu(C))$. However the set $C\setminus C'$ is also $k\Z^d$-invariant and satisfies
\[ \mu(C \setminus C') = \mu(C) - \mu(C') \in \left(0,\frac{\kappa}{2}\right)\]
which contradicts the definition of $\kappa$, proving the claim.
One can then easily check that translates of $C$ by some non-empty subset $J\subset \{v_1,\ldots,v_{k^d}\}$ disjointly cover $X$ up to $\mu$, and $k\Z^d$ acts ergodically on each translate. The result then follows with $\{\mu(\cdot \, | \, T^{v_j}C)\}_{j\in J}$ as the $T^k$-ergodic components of $\mu$.
\end{proof}

\end{document}